\documentclass[12 pt]{amsart}

\usepackage{amsmath}
\usepackage{t1enc}
\usepackage[latin2]{inputenc}
\usepackage{amsmath}
\usepackage{amssymb}
\usepackage{amsthm}
\usepackage{eufrak}
\usepackage{verbatim}
\usepackage{color}
\usepackage{pstricks,pst-node}
\usepackage{enumerate}
\usepackage{xmpmulti}
\usepackage{psfrag}
\usepackage{graphicx}

\newtheorem{thm}{Theorem}[section]
\newtheorem{prop}[thm]{Proposition}
\newtheorem{prob}[thm]{Problem}
\newtheorem{dfn}[thm]{Definition}
\newtheorem{clm}[thm]{Claim}
\newtheorem{lemma}[thm]{Lemma}
\newtheorem{fact}[thm]{Fact}

\newtheorem{cor}[thm]{Corollary}
\newtheorem{question}[thm]{Question}

\newcommand{\omg}{{\omega_1}}
\newcommand{\F}{\varphi}
\newcommand{\B}{\mathcal{W}}
\newcommand{\R}{[\mathbb{R}]^{<\omega}}
\newcommand{\U}{U^\alpha_\gamma}
\newcommand{\V}{U^{<\alpha}_\gamma}

\begin{document}

\title{A counterexample in the theory of $D$-spaces}
\author {D\'aniel T. Soukup \and Paul J. Szeptycki}
\address{Institute of Mathematics, E\"otv\"os Lor\'and University, Budapest, Hungary}
\email{daniel.t.soukup@gmail.com}
\address{
Department of Mathematics and Statistics, York University, Toronto, ON M3J 1P3 Canada}
\email{szeptyck@yorku.ca}

\thanks{The second author acknowledges support from NSERC grant 238944}
\thanks{Corresponding author: D\'aniel T. Soukup}


\begin{abstract}
Assuming $\diamondsuit$, we construct a $T_2$ example of a hereditarily Lindel\"of space of size $\omega_1$ which is not a $D$-space. The example has the property that all finite powers are also Lindel\"of.
\end{abstract}

\maketitle

\section{Introduction} 

The following notion is due to van Douwen, first studied with Pfeffer in \cite{vD}.

\begin{dfn}
 A $T_1$ space $X$ is said to be a {\em $D$-space} if for each open neighbourhood assignment $\{U_x:x\in X\}$ there is a closed and discrete subset $D\subseteq X$ such that $\{U_x:x\in D\}$ covers the space.
\end{dfn}

The question whether every regular Lindel\"of space is $D$ has been attributed to van Douwen \cite{G}. Moreover, van Douwen and Pfeffer pointed out that 

\begin{quote}
"No satisfactory example of a space which is not a $D$-space is known, where by satisfactory example we mean an example having a covering property at least as strong as metacompactness or subparacompactness." 
\end{quote}

Indeed, the lack of satisfactory examples of $D$-spaces satisfying some interesting covering properties continues and there has been quite a bit of activity in the area in the last decades (see the surveys \cite{E} and \cite{G} for other related results and open problems). Whether regular Lindel\"of spaces are $D$-spaces was listed as Problem 14 in Hru\v s\'ak and Moore's list of 20 open problems in set-theoretic topology \cite{HM}, and there are no consistency results in either direction even for hereditarily Lindel\"of spaces. The question whether Lindel\"of implies $D$ for the class of $T_1$ spaces was also open and explicitly asked in \cite{FS} and more recently in \cite{A}. 

In this note, assuming $\diamondsuit$, we construct an example of a hereditarily Lindel\"of $T_2$ space that is not a $D$-space. The example also has the property that every finite power is Lindel\"of, but we do not know if it can be made regular. 

The article is structured as follows; in Section \ref{pre} we gather a few general facts and definitions, and in Section \ref{constr} we present the construction. In Section \ref{further}, we make some remarks and prove further properties of our construction. Finally, in Section \ref{quest} we state a few open problems.

\section{Preliminaries}\label{pre}

 Delicate use of elementary submodels play crucial role in our arguments. We do not intend to give a precise introduction to this powerful tool since elementary submodels are widely used in topology nowadays; let us refer to \cite{Do}. However, we present here a few easy facts and a lemma which could serve as a warm-up exercise for the readers less involved in the use of elementary submodels. 

Let $H(\vartheta)$ denote the sets which have transitive closure of size less than $\vartheta$ for some cardinal $\vartheta$. The following facts will be used regularly without explicitly referring to them.

\begin{fact}
Suppose that $M\prec H(\vartheta)$ for some cardinal $\vartheta$ and $M$ is countable.
\begin{enumerate}[(a)]
 \item If $\mathcal{F}\in M$ and there is $F\in \mathcal{F}\setminus M$ then $\mathcal{F}$ is uncountable.
\item If $B\in M$ and $B$ is countable then $B\subseteq M$.
\end{enumerate}
 
\end{fact}

The next lemma is well-known, nonetheless we present a proof.

\begin{lemma}\label{delta}
 Let $\mathcal{F}\subseteq [\omg]^{<\omega}$ and suppose that $M$ is a countable elementary submodel of $H(\vartheta)$ for some cardinal $\vartheta$ such that $\mathcal{F}\in M$. If there is an $F\in \mathcal{F}$ such that $F\notin M$ then there is an uncountable $\Delta$-system $\mathcal{G}\subseteq \mathcal{F}$ in $M$ with kernel $F\cap M$.

 Moreover, if $\psi(x,...)$ is any formula with parameters from $M$ and $\psi(F,...)$ holds then $\mathcal{G}$ can be chosen in such a way that $\psi(G,...)$ holds for every $G\in \mathcal{G}$, as well.
\end{lemma}
\begin{proof}
Suppose that $\mathcal{F},M,F\in\mathcal{F}\setminus M$ and $\psi$ is as above. Let $D=F\cap M$ and let $\mathcal{F}_0=\{G\in \mathcal{F}:D\subseteq G \text{ and } \psi(G,...)\}$. Clearly, $F\in \mathcal{F}_0\in M$ and $F\notin M$ thus $\mathcal{F}_0$ is uncountable. Moreover, $F\cap \alpha = D$ for all $\alpha$ in a tail of $M\cap \omega_1$; that is, $\exists G\in \mathcal{F}_0 :G\cap \alpha=D$ and this holds in $M$ as well, by elementary. Thus $$M\models \exists \beta<\omg \forall \alpha\in(\beta,\omg)\exists G\in \mathcal{F}_0: G\cap \alpha=D.$$
Thus this holds in $H(\vartheta)$ as well, by elementary. Hence we can select inductively an uncountable $\Delta$-system from $\mathcal{F}_0$. Using elementary again, there is such a $\Delta$-system in $M$ too.

\end{proof}

For any set-theoretic notion, including background on $\diamondsuit$, see \cite{K}.\\

There are different conventions in general topology whether to add regularity to the definition of a Lindel\"of space. In this article, any topological space $X$ is said to be \emph{Lindel\"of} iff every open cover has a countable subcover; that is, no separation is assumed.

 Finally, we need a few other definitions. 

\begin{dfn}
A collection ${\mathcal U}$ of subsets of a space $X$ is called an 
{\em $\omega$-cover} if for every finite $F\subseteq X$ there is $U\in {\mathcal U}$ such that $F\subseteq U$. 
\end{dfn}

\begin{dfn}
  A collection ${\mathcal N}$ of subsets of a space $X$ is called a {\em local $\pi$-network at the point $x$} if for each open neighbourhood $U$ of $x$ in $X$, there is an $N\in {\mathcal N}$ such that $N\subseteq U$ (it is not required that the sets in ${\mathcal N}$ be open, nor that they contain the point $x$).  
\end{dfn}

 \section{The Construction} \label{constr}

We construct a topology by constructing a sequence $\{U_\gamma:\gamma<\omega_1\}$ of subsets of $\omega_1$ such that $\gamma\in U_\gamma$ for every $\gamma\in \omg$. The example will be obtained by first taking the family $\{U_\gamma:\gamma<\omega_1\}$ as a subbasis for a topology on $\omg$ and then refining it with a Hausdorff topology of countable weight.

The following lemma will be used to prove the Lindel\"of property.

\begin{lemma}\label{mainlemma}
 Consider a topology on $\omg$ generated by a family $\{U_\gamma:\gamma<\omg\}$ as a subbase; sets of the form $U_F=\bigcap\{U_\gamma:\gamma \in F\}$ for $F\in [\omg]^{<\omega}$ form a base. If for every uncountable family $B\subseteq [\omg]^{<\omega}$ of pairwise disjoint sets there is a countable $B'\subseteq B$ such that $$|\omg\setminus\bigcup\{U_F:F\in B'\}|\leq\omega$$ then the topology is hereditarily Lindel\"of.
\end{lemma}
\begin{proof}
 Fix an open family $\mathcal{U}$; we may assume that $\mathcal{U}=\{U_F:F\in \mathcal{G}\}$ for some $\mathcal{G}\subseteq [\omg]^{<\omega}$. Let $M$ be a countably elementary submodel of $H(\vartheta)$ for some sufficiently large  $\vartheta$ such that $\{U_\gamma:\gamma\in \omega_1\},\ {\mathcal G}\in M$. 

It suffices to prove that $\bigcup \mathcal{U}$ is covered by the countable family $\mathcal{V}=\{U_F:F\in \mathcal{G}\cap M\}$. $\mathcal{V}$ clearly covers $\bigl(\bigcup \mathcal{U} \bigr)\cap M$ thus we consider an arbitrary $\alpha\in \bigl(\bigcup \mathcal{U} \bigr)\setminus M$. There is $G_0\in \mathcal{G}$ such that $\alpha\in U_{G_0}$; let $F=G_0\cap M$. If $F=G_0$ then $\mathcal{V}$ covers $\alpha$ thus we are done. Otherwise there is an uncountable $\Delta$-system $\mathcal{D}\subseteq \mathcal{G}$ in $M$ with kernel $F$ by Lemma \ref{delta}. Consider the uncountable, pairwise disjoint family $B=\{G\setminus F: G\in \mathcal{D}\}$; by our hypothesis there is a countable $B'\subseteq B$ such that
$$|\omg\setminus\bigcup\{U_H:H\in B'\}|\leq\omega.$$
$B'$ can be chosen in $M$ since $B\in M$; note that $B'\subseteq M$. Thus the countable set of points not covered also lie in $M$. Therefore, there is $G\setminus F\in B'$ such that $\alpha\in U_{G\setminus F}$. Hence $\alpha\in U_G$ since $\alpha\in U_{G_0} \subseteq U_{F}$ and $U_G=U_{G\setminus F}\cap U_F$. This completes the proof of the lemma.
\end{proof}

Let us define now the topology which will be used to ensure the Hausdorff property.

\begin{dfn}
 Define a topology on $\R$ as follows. Let $Q\subseteq \mathbb{R}$ be a Euclidean open set and let $Q^*=\{H\in\R: H\subseteq Q \}$. Sets of the form $Q^*$ define a $\rho$ topology on $\R$.
\end{dfn}

The proof of the following claim is straightforward.

 \begin{clm}
  \begin{enumerate}
 \item $(\R,\rho)$ is of countable weight,
\item any family $\mathcal{X}\subseteq \R$ of pairwise disjoint nonempty sets forms a Hausdorff subspace of $(\R,\rho)$.
\end{enumerate}
 \end{clm}

Let us fix the countable base $\B=\{Q^*:Q$ is a disjoint union of finitely many intervals with rational endpoints$\}$ for $(\R,\rho)$.\\

 For the remainder of this section we suppose $\diamondsuit$; thus $2^\omega=\omg$ and we can fix an enumeration 
\begin{itemize}
 \item $ \{C_\alpha\}_{\alpha<\omg}=[\omg]^{\leq\omega} \text{ such that } C_\alpha\subseteq \alpha \text{ for all } \alpha<\omg.$
\end{itemize}

 Also, there is a $\diamondsuit$-sequence $\{B_\gamma\}_{\gamma<\omg}$  on $[\omg]^{<\omega}$; that is, 

\begin{itemize}
\item $\text{ for every uncountable } B\subseteq [\omg]^{<\omega}
\text{ there are stationary many } \beta\in \omg \text{ such that } B\cap [\beta]^{<\omega}=B_\beta. $
\end{itemize}

The next theorem is the key to our main result; we encourage the reader to first skip the quite technical proof of Theorem \ref{mainthm} and go to Corollary \ref{maincor} to see how our main result is deduced. In particular, \textbf{IH}(3) assures the space is not a $D$-space and \textbf{IH}(4) makes the space hereditarily Lindel\"of.                                                                                                                                                                                                                                                                 
                                                                                                                                                                                                                                                                      \begin{thm}\label{mainthm}
 There exist $\{U^\alpha_\gamma\}_{\gamma\leq\alpha}$ and $\F_\alpha: (\alpha+1) \rightarrow \R$ for $\alpha<\omg$ with the following properties:
\begin{enumerate}[\textbf{IH}(1)]
 \item $U^\alpha_\gamma\subseteq \alpha+1$ and $U^\alpha_\alpha=\alpha+1$ for every $\gamma\leq\alpha<\omg$, and  the family $\varphi_\alpha[\alpha+1]$ is pairwise disjoint for every $\alpha<\omg$.
\item $U^\alpha_\gamma=U^{\alpha_0}_\gamma\cap (\alpha+1)$ and $\F_\alpha=\F_{\alpha_0}\upharpoonright{(\alpha+1)}$ for all $\gamma\leq\alpha\leq\alpha_0$.
\end{enumerate}

 Let $\tau_\alpha$ denote the topology generated by the sets  $$\{U^\alpha_\gamma:\gamma\leq\alpha\}\cup\{\F_\alpha^{-1}(W):W\in \B\}$$ as a subbase. Let $U^\alpha_F=\bigcap\{\U:\gamma \in F\}$ for $F\in [\alpha+1]^{<\omega}$.

\begin{enumerate}[\textbf{IH}(1)]
\setcounter{enumi}{2}
 \item If $C_\alpha$ is $\tau_\alpha$ closed discrete then $\bigcup \{\U:\gamma\in C_\alpha\}\neq\alpha+1$.

\item Let $T_\alpha=\{\beta\leq\alpha:$ $B_\beta$ is a pairwise disjoint family of finite subsets of $\beta$ and there is a countable elementary submodel  $M \prec H(\vartheta)$ for some sufficiently large $\vartheta$ such that (i)-(v) holds from below $\}$.

\begin{enumerate}[(i)] 
 \item $ M\cap\omg=\beta$,
\item $\B,\{B_\gamma\}_{\gamma<\omg}\in M$,
\item there is a function $\F\in M$ such that $\F\upharpoonright\beta=\F_\beta\upharpoonright\beta$,
\item there is an uncountable $B\in M$ such that $M\cap B=B_\beta$, and
\item there is a $\{V_\gamma\}_{\gamma<\omg}\in M$ such that $V_\gamma\cap\beta=U^\beta_\gamma\cap \beta$ for all $\gamma<\beta$.
\end{enumerate}

Then
\begin{enumerate}[(a)]
 \item if $\beta\in T_\alpha$ then $B_\beta$ is a local $\pi$-network at $\beta$ in $\tau_\alpha$,
\item if $\beta \in T_\alpha\cap \alpha$ then for every $V\in \tau_\alpha$ with $\beta\in V$ the family $$\{U^\alpha_F:F\in B_\beta, F\subseteq V\}$$ is an $\omega$-cover of $(\beta,\alpha]$.
\end{enumerate}

\end{enumerate}

\end{thm}

\begin{proof}
 We prove by induction on $\alpha<\omg$ with inductional hypothesises \textbf{IH}(1)-\textbf{IH}(4)! Suppose we constructed $\{U^\beta_\gamma\}_{\gamma\leq\beta}$ for $\beta<\alpha$. Let $$\V=\cup\{U^\beta_\gamma:\gamma\leq\beta<\alpha\} \text{ and } \F_{<\alpha}=\cup\{\F_\beta:\beta<\alpha\}.$$
 Let $\tau^-_\alpha$ denote the topology on $\alpha$ generated by the sets $$\{\V:\gamma<\omg\}\cup\{\F^{-1}_{<\alpha}(W):W\in \B \}$$ as a subbase. Let $U^{<\alpha}_F=\bigcap\{\V:\gamma \in F\}$ for $F\in [\alpha]^{<\omega}$. Note that if $\beta\in T_\alpha\cap \alpha$ then $\beta\in T_{\alpha'}$ for every $\alpha'\in(\beta,\alpha)$; hence \textbf{N}(i) and \textbf{N}(ii) below holds by \textbf{IH}(4):
\begin{enumerate}[\textbf{N}(i)]
 \item for every $V\in\tau^-_\alpha$ with $\beta\in V$ the family $$\{U^{<\alpha}_F: F\in B_\beta, F\subseteq V\}$$ is an $\omega$-cover of $(\beta,\alpha)$;
\item $B_\beta$ is a local $\pi$-network at $\beta$ in $\tau^-_\alpha$.
\end{enumerate}

Therefore, it suffices to define $U^\alpha_\alpha=\alpha+1$ and 
\[ \U = \left\{ \begin{array}{ll}
\V \text{ or } \\
\V \cup \{\alpha\} \end{array} \right. \]

and $\F_\alpha\upharpoonright{\alpha}=\F_{<\alpha}$, $\F_\alpha(\alpha)=x_\alpha\in \R$ such that
\begin{enumerate}[\textbf{D}(i)]
 \item $x_\alpha$ is disjoint from $x_\beta=\F_{<\alpha}(\beta)$ for all $\beta<\alpha$,
\item if $\beta \in T_\alpha\cap\alpha$ then for every $\beta\in V\in \tau_\alpha$ the family $$\{U^\alpha_F:F\in B_\beta, F\subseteq V\}$$ is an $\omega$-cover of $(\beta,\alpha]$,
\item if $C_\alpha$ is $\tau_\alpha^-$ closed discrete then $\alpha\notin \U$ for all $\gamma\in C_\alpha$,
\item if $\alpha\in T_\alpha$ then $B_\alpha$ is a local $\pi$-network at $\alpha$ in $\tau_\alpha$.
\end{enumerate}

\vspace{0.6 cm}
\begin{large}
\textbf{Case I.} $T_\alpha\cap \alpha = \emptyset$
\end{large}

\vspace{0.6 cm}

Let $U^\alpha_\alpha=\alpha+1$, $\U=\V$ for $\gamma <\alpha$. We proceed differently according to whether $\alpha\notin T_\alpha$ or $\alpha\in T_\alpha$.

\vspace{0.6 cm}
\begin{large}
 \textbf{Subcase A.} $\alpha\notin T_\alpha$
\end{large}
\vspace{0.6 cm}

Pick any $x_\alpha\in \R$ disjoint from $x_\beta$ for all $\beta<\alpha$. Clearly, \textbf{D}(i)-\textbf{D}(iv) are satisfied.

\vspace{0.6 cm}
\begin{large}
\textbf{Subcase B.} $\alpha\in T_\alpha$
\end{large}
\vspace{0.6 cm}

It is clear that \textbf{D}(ii) and \textbf{D}(iii) are satisfied. Let $M$ be a countable elementary submodel of $H(\vartheta)$ for some sufficiently large $\vartheta$ showing that $\alpha\in T_\alpha$. To find the appropriate $x_\alpha\in \R$ we need that $B_\alpha$ is a local $\pi$-network at $\alpha$ in $\tau_\alpha$. Since $\{\F^{-1}_\alpha(W):W\in \B,x_\alpha\in W\}$ will be a base at $\alpha$ in $\tau_\alpha$ we need that for all $W\in \B$ such that $x_\alpha\in W$ there is an $F\in B_\alpha$ such that $\F[F]\subseteq W$. Since $\F[F]\subseteq W$ iff $\cup\F[F]\in W$, we need to find an accumulation point of the finite sets $\{\cup\F[F]:F\in B_\alpha\}$. We prove the following which will suffice:

\begin{clm}\label{clmB} 
 There is an $x_\alpha\in \R$ such that $x_\alpha\cap x_\beta=\emptyset$ for all $\beta<\alpha$ and for all $W \in \B$ such that $x_\alpha\in W$ we have
$$M \models |\{F\in B: \cup\F[F]\in W\}|>\omega.$$
\end{clm}

Indeed, \textbf{D}(i) is satisfied. Let us check \textbf{D}(iv);  clearly, $$M\models \exists F\in B: \cup\F[F]\in W$$ for every $W\in \B$ such that $x_\alpha\in W$. Hence there is $F\in B\cap M=B_\alpha$ such that $\cup\F[F]=\cup\F_\alpha[F]\in W$, that is $\F_\alpha[F]\subseteq W$. Thus $B_\alpha$ is a local $\pi$-network at $\alpha$ in $\tau_\alpha$; that is, \textbf{D}(iv) is satisfied.\\

\begin{proof}[Proof of Claim \ref{clmB}]
Since $M\models |B|>\omega$ there is $\widetilde{B}\in[B]^\omg \cap M$ and $k\in\omega, \{n_i:i<k\}\subseteq \omega$ such that $|F|=k$ for all $F\in \widetilde{B}$ and if $F=\{\gamma_i:i<k\}$ then $|\F(\gamma_i)|=n_i$ for all $i<k$.

Let $s=\sum_{i<k} n_i$. Now consider the pairwise disjoint $s$-element subsets $\{\cup\F[F]:F\in \widetilde{B}\}$ in $\mathbb{R}$. Clearly 
\begin{multline*}
M\models \bigl(\text{there are uncountably many pairwise disjoint } x\in [\mathbb{R}]^s \text{ such that }\\ |\{F\in \widetilde{B}:\cup\F[F]\in W\}|>\omega \text{ for every } W\in \B \text{ with } x\in W \bigr). 
\end{multline*}
 Hence, there is $x_\alpha\in [\mathbb{R}]^s$ disjoint from $x_\beta$ for all $\beta<\alpha$ such that $|\{F\in \widetilde{B}:\cup\F[F]\in W\}|>\omega$ for all $W\in \B$ with $x_\alpha\in W$. Thus $$M\models |\{F\in \widetilde{B}:\cup\F[F]\in W\}|>\omega$$
which we wanted to prove.
\end{proof} 

\vspace{0.6 cm}
\begin{large}
\textbf{Case II.} $T_\alpha\cap \alpha \neq \emptyset$
\end{large}
\vspace{0.6 cm}

Let $T_\alpha\cap\alpha=\{\beta_n:n\in\omega\}$ and $\{G_n:n\in\omega\}\subseteq [\alpha]^{<\omega}$ such that for all $\beta\in T_\alpha\cap \alpha$ and $G\subseteq (\beta,\alpha)$ there are infinitely many $n\in\omega$ such that $\beta=\beta_n$ and $G=G_n$. Let $\{V_k(\beta):k<\omega\}$ denote a decreasing neighbourhood base for the point $\beta\in T_\alpha\cap\alpha$ in $\tau^-_\alpha$. Note that $\{V_n(\beta_n):n\in\omega, \beta_n=\beta\}$ is a base for $\beta\in T_\alpha\cap\alpha$.

\vspace{0.6 cm}
\begin{large}
 \textbf{Subcase A.} $\alpha\notin T_\alpha$
\end{large}
\vspace{0.6 cm}

We need the following claim:

\begin{clm} There is $F_n\in B_{\beta_n}$ for $n\in\omega$ such that 
\begin{enumerate}[\textbf{A}(i)]
 \item $F_n\subseteq V_n(\beta_n)$,
\item $G_n\subseteq U^{<\alpha}_{F_n}$,
\item $F_n\cap C_\alpha =\emptyset$ if $C_\alpha$ is $\tau^-_\alpha$ closed discrete,
\end{enumerate}
\end{clm}
\begin{proof}
There is $V\in \tau_\alpha^-$ such that $\beta_n\in V\subseteq V_n(\beta_n)$ and if $C_\alpha$ is closed discrete, then $C_\alpha\cap V\subseteq \{\beta_n\}$. The family $$\{U^{<\alpha}_F: F\in B_{\beta_n}, F\subseteq V\}$$ is an $\omega$-cover of $(\beta_n,\alpha)$ by \textbf{N}(ii), thus there is $F_n\in B_{\beta_n}$ such that $F_n\subseteq V$ and $G_n\subseteq U^{<\alpha}_{F_n}$.
\end{proof}

Let $U^\alpha_\alpha=\alpha+1$ and for $\gamma<\alpha$ let
\[ \U = \left\{ \begin{array}{ll}
\V & \mbox{if $\gamma\notin \cup\{F_n:n\in\omega\}$, } \vspace{0.2 cm} \\ 
\V \cup \{\alpha\}  & \mbox{if $\gamma\in \cup\{F_n:n\in\omega\}$. } \end{array} \right. \]

Pick any $x_\alpha\in \R$ disjoint from $x_\beta$ for all $\beta<\alpha$.

\textbf{D}(i), \textbf{D}(iii), and \textbf{D}(iv) are trivially satisfied. Let us check \textbf{D}(ii); fix $\beta\in T_\alpha\cap \alpha$, any neighbourhood $V\in \tau_\alpha$ such that $\beta\in V$, and a finite subset $G \subseteq (\beta,\alpha)$. We show that there is an $F\in B_\beta$, such that $U^\alpha_F$ covers $G\cup\{\alpha\}$ and $F\subseteq V$. There is $n\in\omega$ such that $\beta_n=\beta$, $G_n=G$, and $V_n(\beta_n)\subseteq V$; then $F_n\in B_{\beta_n}$ and $F=F_n$ does the job by  \textbf{A}(i), \textbf{A}(ii) and the fact that $\alpha\in U^\alpha_{F_n}$.

\vspace{0.6 cm}

\begin{large}
\textbf{Subcase B.} $\alpha\in T_\alpha$ 
\end{large}

\vspace{0.6 cm}

Let $M$ be a countable elementary submodel of $H(\vartheta)$ for some sufficiently large $\vartheta$ showing that $\alpha\in T_\alpha$. Since $M\models |B|>\omega$ there is $\widetilde{B}\in[B]^\omg \cap M$ and $k\in\omega, \{n_i:i<k\}\subseteq \omega$ such that $|F|=k$ for all $F\in \widetilde{B}$ and if $F=\{\gamma_i:i<k\}$ then $|\F(\gamma_i)|=n_i$ for all $i<k$. Let $s=\sum_{i<k}n_i$. Enumerate $\alpha$ as $\{\alpha_n:n\in\omega\}$.

\begin{clm}There are $F_n\in B_{\beta_n}$ and $W_n\in \B$ for $n\in\omega$ such that 
\begin{enumerate}[\textbf{B}(i)]
 \item $F_n\subseteq V_n(\beta_n)$,
\item $G_n\subseteq U^{<\alpha}_{F_n}$,
\item $F_n\cap C_\alpha =\emptyset$ if $C_\alpha$ is $\tau^-_\alpha$ closed discrete,
\item $W_n=\bigl(\cup\{Q_{n,i}:i<s\}\bigr)^*$ is a basic open set of the topology $\rho$ corresponding to $s$ many disjoint rational intervals $\{Q_{n,i}:i<s\}$ of diameter less than $\frac{1}{n}$,
\item $\overline{Q_{n+1,i}}\subseteq Q_{n,i}$ for every $i<s$ in the Euclidean topology,
\item $\F(\alpha_n)$ is disjoint from $\cup\{Q_{n,i}:i<s\}$, 
\item  and finally $$M\models |\bigl\{F\in \widetilde{B}: F\subseteq \cap\{V_{F_k}:k\leq n\} \text { and } \cup\F[F] \in W_n\bigr\}|>\omega.$$
\end{enumerate}
\end{clm}
\begin{proof}
We construct $F_n$ and $W_n$ by induction on $n\in\omega$. Suppose we constructed $F_k$ and $W_k$ for $k<n$ such that the hypothesises \textbf{B}(i)-\textbf{B}(vii) above are satisfied.

 Let $D=\bigl\{F\in \widetilde{B}:F\subseteq \cap\{V_{F_k}:k< n\} \text { and } \cup\F[F] \in W_{n-1}\bigr\}$ if $n>0$ and $D=\widetilde{B}$ if $n=0$; then $M\models |D|>\omega$. Just as in Claim \ref{clmB}
\begin{multline*}
M\models \bigl(\text{there are uncountably many pairwise disjoint } x\in \mathbb{R}^s\cap W_{n-1} \text{ such that } \\ |\{F\in D:\cup\F[F]\in W\}|>\omega \text{ for every } W\in \B \text{ with }x\in W\bigr).
\end{multline*}
 Choose $x\in \mathbb{R}^s\cap W_{n-1}$ such that $|\{F\in D:\cup\F[F]\in W\}|>\omega \text{ for every } W\in \B$ with $x\in W$ and $x\cap \F(\alpha_n)=\emptyset$. Let $x=\{x_i:i<s\}$ and choose $W_n=\bigl(\cup\{Q_{n,i}:i<s\}\bigr)^*$ such that 
\begin{itemize}
\item $Q_{n,i}$ is a rational interval of diameter less then $\frac{1}{n}$ for every $i<s$,
 \item $x_i\in Q_{n,i}$ for every $i<s$,
\item $\overline{Q_{n,i}}\subseteq Q_{n-1,i}$ for every $i<s$ in the Euclidean topology (if $n>0$),
\item $\cup \{Q_{n,i}:i<s\} \cap \F(\alpha_n)=\emptyset$.
\end{itemize}

Let $D'=\bigl\{F\in \widetilde{B}:F\subseteq \cap\{V_{F_k}:k< n\},\cup\F[F] \in W_{n}  \text { and } \beta_n< \min F\bigr\}$; clearly, $M\models |D'|>\omega$. Let $V\in \tau_\alpha^-$ such that $\beta_n\in V\subseteq V_n(\beta_n)$ and $V\cap C_\alpha \subseteq \{\beta_n\}$ if $C_\alpha$ is $\tau_\alpha^-$ closed discrete. Applying \textbf{N}(i) to $F\cup G_n$ for $F\in D'\cap M$ gives us that there is $F_n(F)\in B_{\beta_n}$ such that $F_n(F)\subseteq V$ and $$U^{<\alpha}_{F_n(F)}\text{ covers } F\cup G_n$$ and hence  $V_{F_n(F)}$ covers $F\cup G_n$. Thus 
\begin{multline*}
M\models \bigl(\text{ for every } F\in D' \text{ there is } F_n(F)\in B_{\beta_n} \text{ such that }\\ F_n(F)\subseteq V \text{ and } V_{F_n(F)} \text{ covers } F\cup G_n.\bigr)
\end{multline*}
 Finally, note that $M\models |B_{\beta_n}|\leq \omega$; thus 
\begin{multline*}
 M\models \bigl( \text{there is } F_n\in B_{\beta_n} \text{ such that } F_n\subseteq V \text{ and } V_{F_n} \text{ covers } F\cup G_n \\ \text{ for uncountably many } F\in D'\bigr).\end{multline*}
 It is now easily checked that $F_n$ and $W_n$ satisfies properties \textbf{B}(i)-\textbf{B}(vii).
\end{proof}

Let $U^\alpha_\alpha=\alpha+1$ and for $\gamma<\alpha$ let
\[ \U = \left\{ \begin{array}{ll}
\V & \mbox{if $\gamma\notin \cup\{F_n:n\in\omega\}$,} \vspace{0.2 cm} \\
\V \cup \{\alpha\}  & \mbox{if $\gamma\in \cup\{F_n:n\in\omega\}$. } \end{array} \right. \]

Let $x_\alpha\in \R$ be the unique $s$-element subset of $\mathbb{R}$ in the intersection $\cap\{\cup\{Q_{n,i}:i<s\}:n\in \omega\}$; existence and uniqueness follows from \textbf{B}(iv) and \textbf{B}(v), and $x_\alpha$ is disjoint from $x_\beta$ for all $\beta<\alpha$ by \textbf{B}(vi). Note that $$\bigl\{\cap\{U^\alpha_{F_k}:k\leq n\}\cap \F^{-1}_\alpha(W_n):n\in\omega\bigr\}$$ is a base for the point $\alpha$ in $\tau_\alpha$.

\textbf{D}(i) is satisfied by \textbf{B}(vi) and the fact that $x_\alpha \subseteq \cup\{Q_{n,i}:i<s\}$. 

Let us check \textbf{D}(ii); fix $\beta\in T_\alpha\cap \alpha$, any neighbourhood $V\in \tau_\alpha$ such that $\beta\in V$, and a finite subset $G \subseteq (\beta,\alpha)$. We show that there is an $F\in B_\beta$, such that $U^\alpha_F$ covers $G\cup\{\alpha\}$ and $F\subseteq V$. There is $n\in\omega$ such that $\beta_n=\beta$, $G_n=G$, and $V_n(\beta_n)\subseteq V$; $F_n\in B_{\beta_n}$ does the job by  \textbf{B}(i), \textbf{B}(ii) and the fact that $\alpha\in U^\alpha_{F_n}$.

\textbf{D}(iii) is satisfied by \textbf{B}(iii) and the definition of $U^\alpha_\gamma$.

Finally, let us check \textbf{D}(iv); it suffices to show that for every $n\in \omega$ there is $F\in B_\alpha$ such $$F\subseteq \cap\{U^\alpha_{F_k}:k\leq n\}\cap \F^{-1}_\alpha(W_n).$$ Condition \textbf{B}(vii) gives us this, using the observation that $\F_\alpha[F]\subseteq W$ iff $\cup \F_\alpha[F]\in W$ for any $F\in B_\alpha$ and $W\in \B$.

\vspace{0.6 cm}

By all means, this completes the proof of the theorem.\\
\end{proof}

Now we are ready to deduce our main result.

\begin{cor}\label{maincor} Suppose that $\{U^\alpha_\gamma\}_{\gamma\leq\alpha}$ and $\F_\alpha: (\alpha+1) \rightarrow \R$ for $\alpha<\omg$ are as in Theorem \ref{mainthm} and let $U_\gamma=\cup\{\U:\gamma\leq\alpha<\omg\}$ for $\gamma<\omg$ and $\F=\cup\{\F_\alpha: \alpha<\omg\}$. Let $\tau$ denote the topology on $\omg$ generated by the sets $$\{U_\gamma:\gamma<\omg\}\cup\{\F^{-1}(W):W\in \B\}$$ as a subbase.

The space $(\omg,\tau)$ is hereditarily Lindel\"of, Hausdorff but not a $D$-space. Also, $(\omg,\tau)$ has countable $\Psi$-weight.
\end{cor}
\begin{proof}
 First, we show that $(\omg,\tau)$ is hereditarily Lindel\"of and Hausdorff. We need the following observation.
\begin{clm}\label{lindref}
 A Hausdorff topology of countable weight $\tau_{\text{sc}}$ refined by a hereditarily Lindel\"of topology $\tau_{\text{hl}}$  on some set $X$ is again a hereditarily Lindel\"of, Hausdorff topology on $X$.
\end{clm}
\begin{proof}
Let $\tau_{\text{ref}}$ denote the topology generated by $\tau_{\text{sc}}\cup\tau_{\text{hl}}$ as a subbase; that is, $\tau_{\text{ref}}$ is the common refinement of  $\tau_{\text{sc}}$ and $\tau_{\text{hl}}$. $\tau_{\text{ref}}$ is clearly Hausdorff, we prove that for any open family $\mathcal{U}\subseteq \tau_{\text{ref}}$ there is a countable $\mathcal{U}_0\subseteq \mathcal{U}$ such that $\cup\mathcal{U}_0=\cup\mathcal{U}$. We can suppose that $$\mathcal{U}=\{U_i\cap V^i_j:i\in\omega, j\in I_i\}$$ where $\{U_i:i\in\omega\}\subseteq \tau_{\text{sc}}$ and $\{V^i_j:i\in\omega,j\in I_i\}\subseteq \tau_{\text{hl}}$ for some index sets $\{I_i:i\in\omega\}$. For every $i\in\omega$ there is a countable $J_i\subseteq I_i$ such that $$U_i\cap \bigcup\{V^i_j:j\in I_i\}=U_i\cap \bigcup\{V^i_j:j\in J_i\}$$ by the hereditarily Linel\"ofness of $\tau_{\text{hl}}$. Thus $$\cup\mathcal{U}=\bigcup\{U_i\cap V^i_j:i\in\omega, j\in J_i\}$$ which completes the proof.
\end{proof}

Therefore, it suffices to prove that the topology generated by $\{U_\gamma:\gamma<\omg\}$ as a subbase on $\omg$ is hereditarily Lindel\"of. Lemma \ref{mainlemma} and the proposition below gives us this result. Let $U_F=\bigcap\{U_\gamma:\gamma\in F\}$ for $F\in[\omg]^{<\omega}$.
 
\begin{prop}\label{prop} For any uncountable family of pairwise disjoint sets $B\subseteq [\omega_1]^{<\omega}$, there is a countable $B'\subseteq B$ such that $\{U_F:F\in B'\}$ is a cover, moreover an $\omega$-cover of a tail of $\omega_1$. 
\end{prop}

\begin{proof} Fix some uncountable family $B\subseteq [\omg]^{<\omega}$ of pairwise disjoint sets. There is an $M \prec H(\vartheta)$ for some sufficiently large $\vartheta$ such that $B,\F, \{U_\gamma:\gamma<\omg\},\{B_\gamma:\gamma<\omg\},\B \in M$ and $$M\cap\omg=\beta \text{ and } B\cap M=B\cap [\beta]^{<\omega}=B_\beta.$$
We claim that $\bigcup\{U_F:F\in B'\}$ is an $\omega$-cover of $\omg\setminus(\beta+1)$ for the countable $B'=B_\beta$. Indeed, fix some finite $K\subseteq \omg\setminus(\beta+1)$ and let $\alpha\in\omg\setminus(\beta+1)$ such that $K\subseteq \alpha$. Then $\beta\in T_\alpha$ ensured by the model $M$, and hence there is some $F\in B_\beta=B'$ such that $K\subseteq U^\alpha_F\subseteq U_F$ by \textbf{IH}(4).
\end{proof}

Now we prove that $(\omg,\tau)$ is not a $D$-space. Consider the neighbourhood assignment $\gamma\mapsto U_\gamma$; we show that $\cup\{U_\gamma:\gamma\in C\}\neq\omg$ for every closed discrete $C\subseteq \omg$. Since $(\omg,\tau)$ is Lindel\"of, $|C|\leq \omega$ and hence there is $\alpha<\omg$ such that $C_\alpha=C$. It suffices to note that $C_\alpha$ is $\tau_\alpha$ closed discrete if $\tau$ closed discrete; indeed, then $\cup\{U_\gamma:\gamma\in C_\alpha\}\neq \alpha+1$ by \textbf{IH}(3).\\

Finally, $(\omg,\tau)$ has countable $\Psi$-weight since $\tau$ is a refinement of a Hausdorff topology which is of countable weight.

\end{proof}

\section{Further properties} \label{further}

In \cite{vMTW} the authors asked the following:

\begin{prob}[{\cite[Problem 4.6]{vMTW}}]
 Suppose that a space $X$ has the property that for every open neighbourhood assignment $\{U_x:x\in X\}$ there is a second countable subspace $Y$ of $X$ such that $\bigcup\{U_x:x\in Y\}=X$ (\emph{dually second countable}, in short). Is $X$ a $D$-space?
\end{prob}

Our construction answers this question in the negative.

\begin{prop}
The space $X$ constructed in Corollary \ref{maincor} is dually second countable, however not a $D$-space.
\end{prop}
\begin{proof}
 The space $X$ has the property that every countable subspace is second countable; indeed, the subspace topology on $\alpha\in\omg$ is generated by the sets $U_\beta\cap \alpha$ for $\beta<\alpha$ and $\{\F^{-1}(W):W\in \B\}$, using the notations of the previous section. Therefore, by the Lindel\"of property, for every open neighbourhood assignment there is a countable and hence second countable subspace whose neighbourhoods cover the space.
\end{proof}

Our aim now is to prove that the space constructed in Corollary \ref{maincor} has the property that all its finite powers are Lindel\"of. Indeed, by a theorem of Gerlits and Nagy \cite{GN}, a space has all finite powers Lindel\"of if and only if the space is an $(\varepsilon)$-space, i.e., every $\omega$-cover has a countable $\omega$-subcover. 

Let us call our space from Corollary \ref{maincor} $X$, and now establish the following theorem:

\begin{thm}\label{epsilon} Every subspace of $X$ is an $(\varepsilon)$-space. 
 \end{thm}
\begin{proof}

\begin{comment}

First, we prove an analogue of Claim \ref{lindref}.

\begin{clm}\label{epsilonref}Any hereditarily $(\varepsilon)$-space topology $\tau_0$ refined by a second countable topology $\tau_1$ is again a hereditarily $(\varepsilon)$-space. 
\end{clm}
\begin{proof}
 Suppose that $\tau_0$ and $\tau_1$ are topologies on a set $T$; let $\tau^*$ denote the refinement of $\tau_0$ and $\tau_1$ and fix $S\subseteq T$ and $\omega$-cover $\mathcal{U}$  of $S$. 
\end{proof}

Therefore, it suffices to prove that the topology generated by the sets $\{U_F:F\in[\omg]^{<\omega}\}$ on $\omg$ is a hereditarily $(\varepsilon)$-space. 

\end{comment}

First, let us prove the following analogue of Lemma \ref{mainlemma}.

\begin{lemma}\label{mainlemma2}
 Consider a topology on $\omg$ generated by a family $\{U_\gamma:\gamma<\omg\}$ as a subbase. If for every uncountable family $B\subseteq [\omg]^{<\omega}$ of pairwise disjoint sets there is a countable $B'\subseteq B$ such that $$\{U_F:F\in B'\} \text{ is an } \omega \text{-cover of a tail of }\omg $$ then the topology is a hereditarily $(\varepsilon)$-space.
\end{lemma}
\begin{proof} Fix $Y\subseteq X$ and an $\omega$-cover $\mathcal{U}$ of $Y$; we can suppose that $\mathcal{U}=\{\cup\{U_{F_i}:i<m\}:\{F_i:i<m\}\in \mathbb{F}\}$ for some $\mathbb{F}\subseteq \bigl[[\omg]^{<\omega}\bigr]^{<\omega}$. Let $M$ be a countably elementary submodel of $H(\vartheta)$ for some sufficiently large  $\vartheta$ such that $\{U_\gamma:\gamma\in \omega_1\}, \mathbb{F}\in M$. It suffices to prove the following.
\begin{clm}
 $M\cap\mathcal{U}$ is a countable $\omega$-cover of $Y$.
\end{clm}
\begin{proof}
Let $K\in [Y]^{<\omega}$ and let $L=K\cap M$. Clearly, $M\cap\mathcal{U}$ covers $K$ if $K=L\subseteq M$; thus, we can suppose that $K\neq L$ and hence $K\notin M$. There is some $\{F_i:i<m\}\in \mathbb{F}$ such that $K\subseteq \cup\{U_{F_i}:i<m\}$. Let $D_i=F_i\cap M$ for $i<m$ and we can suppose that there is some $n\leq m$ such that $F_i\neq D_i$ for $i<n$ and $F_i=D_i$ for $n\leq i < m$. It follows from Lemma \ref{delta} that there is an uncountable sequence $\{\{F^\alpha_i:i<m\}:\alpha<\omg\}\subseteq \mathbb{F}$ in $M$ such that 
\begin{enumerate}
 \item $\{F^\alpha_i:\alpha<\omg\}$ is an uncountable $\Delta$-system with kernel $D_i$ for every $i<n$,
\item $F^\alpha_i=F_i$ for all $\alpha<\omg$ and $n\leq i < m$,
\item $\beta\in U_{F_i}$ iff $\beta\in U_{F^\alpha_i}$ for every $\beta\in L$ and $\alpha<\omg$, $i< m$.
\end{enumerate}
The uncountable family $\{F^\alpha_i\setminus D_i:\alpha<\omg\}$ is pairwise disjoint for every $i<n$. Hence if we let $F^\alpha=\bigcup_{i<n} (F^\alpha_i\setminus  D_i)$ for $\alpha<\omg$ then there is $\Theta\in [\omg]^\omg\cap M$ such that $B=\{F^\alpha:\alpha\in \Theta\}$ is pairwise disjoint as well. By our hypothesis and elementary of $M$ there is $J\in [\Theta]^\omega\cap M$ such that the countable  $B'=\{U_{F^\alpha}:\alpha\in J\}$ is an $\omega$-cover of a tail of $\omg$; hence, an $\omega$-cover of $\omg\setminus M$ since finite sets which are not covered also lie in $M$. So there is $\alpha\in J$, and hence $\alpha\in M$, such that $K\setminus M \subseteq U_{F^\alpha}=\bigcap_{i<n}U_{F^\alpha_i\setminus D_i}$. The open set $U=\cup\{U_{F^\alpha_i}:i<m\}$ is in $\mathcal{U}\cap M$. 

We claim that $K\subseteq U$. Fix $\beta\in K$; there is some $i<m$ such that $\beta \in U_{F_i}$. If $\beta\in L$ then $\beta\in U_{F^\alpha_i}$ by (3). Suppose that $\beta\in K\setminus M$; if $n\leq i < m$ then $F_i=F^\alpha_i$ and we are done. If $i<n$ then $\beta\in U_{F^\alpha_i\setminus D_i}$ and $\beta\in U_{F_i}\subseteq U_{D_i}$ so $\beta\in U_{F^\alpha_i}$.  
\end{proof} 
We are done with the proof of Lemma \ref{mainlemma2}.
\end{proof}

We claim that Proposition \ref{prop} and Lemma \ref{mainlemma2} implies that $X$ is a hereditarily $(\varepsilon)$-space. Indeed, our topology $\tau$ on $\omg$ is generated by the sets $$\{U'_\delta:\delta<\omg\}=\{U_\gamma: \gamma<\omg\}\cup\{\F^{-1}(W):W\in \B\}$$
as a subbase. Suppose that $B\subseteq [\omg]^{<\omega}$ is an uncountable family of pairwise disjoint sets; since $\B$ is countable, there is $B_0\in [B]^\omg$ such that $U'_\delta\in \{U_\gamma:\gamma\in \omg\}$ for $\delta\in \cup B_0$. Thus by Proposition \ref{prop}, there is some countable $B'\subseteq B_0$ such that  $\{U'_F:F\in B'\}$ is an $\omega$-cover of a tail of $\omega_1$. Hence the assumption of Lemma \ref{mainlemma2} holds for $X$, thus $X$ is a hereditarily $(\varepsilon)$-space. 
\end{proof}

A well known weakening of $\diamondsuit$ is Ostaszewski's $\clubsuit$, which is known to be consistent with $\omg<2^\omega$. We remark that $\clubsuit$ is not enough to construct a space of size $\omg$ which is Hausdorff and Lindel\"of but not a $D$-space.

\begin{clm}
 It is consistent that $\clubsuit$ holds, $2^\omega$ is arbitrarily large, and every $T_1$ Lindel\"of space of size less than $2^\omega$ is a $D$-space. 
\end{clm}
\begin{proof}
It is known that $T_1$ Lindel\"of spaces of size less than the dominating number $\mathfrak{d}$ are Menger, and L. Aurichi proved that every Menger space is a $D$-space \cite{aurichi}. Thus, it suffices to show that there is a model of ZFC where $\clubsuit$ holds, $2^\omega$ is arbitrarily large, and $\mathfrak{d}=2^\omega$. I. Juh\'asz proved in an unpublished note that it is consistent that $\clubsuit$ holds, $2^\omega$ is arbitrarily large, and Martin's Axiom holds for countable posets; for a proof see \cite{fshs}. It is easy to see that Martin's Axiom for countable posets imply $\mathfrak{d}=2^\omega$.
\end{proof}

\section{Questions} \label{quest}

Let us state some questions concerning Theorem \ref{epsilon}. We do not know whether the analogue of Claim \ref{lindref} holds for hereditarily $(\varepsilon)$-spaces.

\begin{question}
 Suppose that $\tau$ and $\sigma$ are second countable and hereditarily $(\varepsilon)$-space topologies respectively on some set $X$. Is the topology generated by $\tau\cup\sigma$ a hereditarily $(\varepsilon)$-space again?
\end{question}

We do not know if being a hereditarily $(\varepsilon)$-space implies the hereditarily Lindel\"ofness of finite powers.

\begin{question}
 Suppose that a space $X$ is a hereditarily $(\varepsilon)$-space. Is $X^n$ hereditarily Lindel\"of for all $n\in\omega$?
 \end{question}

The following might be easier, nonetheless seems to be open.

\begin{question}
 Suppose that a space $X$ has the property that $A^2$ is Lindel\"of for all $A\subseteq X$. Is $X^2$ hereditarily Lindel\"of?  
\end{question}

We mention two other versions of the question above.

\begin{question}
 \begin{enumerate}[(i)]
  \item Suppose that a space $X$ has the property that $A\times B$ is Lindel\"of for all $A,B\subseteq X$. Is $X^2$ hereditarily Lindel\"of? 
\item  Suppose that the spaces $X,Y$ have the property that that $A\times B$ is Lindel\"of for all $A\subseteq X$ and $B\subseteq Y$. Is $X\times Y$ hereditarily Lindel\"of?
 \end{enumerate}

\end{question}

Of course, the main interest is in obtaining a regular counterexample to van Douwen's question. We conjecture that one should be able to modify our construction in such a way that the sets $\{U_\gamma:\gamma\in \omega_1\}\cup \{\omega_1\setminus U_\gamma:\gamma\in \omega_1\}$ generate a $0$-dimensional, $T_1$ topology that is not a $D$-space and has some additional interesting covering properties. E.g., 

\begin{question} Can we modify the construction to obtain a $0$-dimensional $T_1$ (hence regular) Lindel\"of non $D$-space? 
\end{question}

Finally, let us finish with a more general question. We believe that our construction can be modified so that its finite powers are hereditarily Lindel\"of, thus its $\omega$th power as well.

\begin{question}
 Suppose that a regular space $X$ has the property that $X^\omega$ is hereditarily Lindel\"of. Is $X$ a $D$-space?
\end{question}

\end{document}